\documentclass[10pt]{article}
\usepackage{amssymb,amsmath, amsthm,bbm, mathrsfs,verbatim}
\usepackage{fullpage}
\theoremstyle{theorem}
\newtheorem{thm}{Theorem}
\newtheorem{cor}[thm]{Corollary}
\newtheorem{prop}[thm]{Proposition}

\newtheorem{defn}[thm]{Definition}
\newtheorem{lem}[thm]{Lemma}

\newcommand{\supp}{\operatorname{supp}}

\begin{document}

\author{Emily J. King$^1$\\
$^1$Technische Universit\"at Berlin\\
Stra\ss e des 17. Juni 136\\
10623 Berlin, Germany\\
email: king@math.tu-berlin.de}
\title{Smooth {P}arseval frames for {$L^2(\mathbb{R})$} and generalizations to {$L^2(\mathbb{R}^d)$}}
\date{October 29, 2012}

\maketitle

\abstract{Wavelet set wavelets were the first examples of wavelets that may not have associated multiresolution analyses.  Furthermore, they provided examples of complete orthonormal wavelet systems in $L^2(\mathbb{R}^d)$ which only require a single generating wavelet.  Although work had been done to smooth these wavelets, which are by definition discontinuous on the frequency domain, nothing had been explicitly done over $\mathbb{R}^d$, $d >1$.  This paper, along with another one cowritten by the author, finally addresses this issue. Smoothing does not work as expected in higher dimensions.  For example, Bin Han's proof of existence of Schwartz class functions which are Parseval frame wavelets and approximate Parseval frame wavelet set wavelets does not easily generalize to higher dimensions.  However, a construction of wavelet sets in $\widehat{\mathbb{R}}^d$ which may be smoothed is presented.  Finally, it is shown that a commonly used class of functions cannot be the result of convolutional smoothing of a wavelet set wavelet.
\\[2 pt]{\bf AMS Subject Classification:} 42C40, 42B99, 46A11
\\[2 pt]{\bf Key Words and Phrases:} Parseval wavelet sets, smoothing, frame bound gaps, partitions of unity, Schwartz functions
}

%%%%%%%%%%%%%%%%%%%%%%%%%
\section{Introduction}
%%%%%%%%%%%%%%%%%%%%%%%%%
\subsection{Motivation}
%%%%%%%%%%%%%%%%%%%%%%%%%
For many years it was thought that a single wavelet could not generate an orthonormal basis for $L^2(\mathbb{R}^d)$, $d>1$; however, the groundbreaking work of Dai, Larson, and Speegle \cite{DaL98} \cite{DLS97} \cite{DLS98} and Hern\'{a}ndez, Wang, and Weiss \cite{HWW96} \cite{HWW97} introduced \emph{wavelet sets}, which provided a counter-example to this common belief.  However, with the exception of the recent paper \cite{Mer08}, all constructed wavelet sets for $\mathbb{R}^d$, $d>2$, yielded wavelets with very complicated, fractal-like spectral support  \cite{BMM99} \cite{BLe99} \cite{BLe01} \cite{DLS97} \cite{DLS98} \cite{SW98} \cite{Zak96}.  Wavelet set wavelets have, by definition, discontinuous Fourier transforms and so do not even lie in $L^1(\mathbb{R}^d)$.  A number of successful attempts, some constructive, have been made to smooth $1$-dimensional wavelet set wavelets \cite{BJMP} \cite{ChH97} \cite{BHn94} \cite{BHn97} \cite{HWW96} \cite{HWW97} (see \cite{Dissert} for a summary of the various results).  A systematic construction of (non-orthogonal) Parseval frame wavelet sets without complicated spectral support may be found in \cite{BSu02}.  In \cite{BenKing09}, these wavelets were smoothed on the frequency domain by convolution with elements of approximate identities, yielding frame wavelets which converged in $L^p$ to the original Parseval frame wavelets. Not only did this seemingly natural method yield unexpected results, namely the resulting frame bounds were uniformly bounded away from one (called \emph{frame bound gaps}), but the error worsened as the dimension of the wavelet set increased.  In this paper, we attempt to generalize Bin Han's non-constructive proof of the existence of Schwartz class functions which approximate Parseval wavelet set wavelets in $L^2(\mathbb{R})$ to $L^2(\mathbb{R}^d)$.  We show that the natural approaches to such a generalization fail.  These results show that smoothing of wavelet sets does not trivially generalize to higher dimensions.  Furthermore, we show that a collection of well-known functions which also approximate wavelet set wavelets generate frames with upper frame bounds that converge to $1$ and thus cannot result from convolutional smoothing by an approximate identity.
%%%%%%%%%%%%%%%%%%%%
\subsection{Preliminaries}
%%%%%%%%%%%%%%%%%%%%%
We begin by defining the key mathematical items of interest.
\begin{defn}
A sequence $\lbrace e_{j} \rbrace _{j\in J}$ in a separable Hilbert space $\mathcal{H}$ is a \emph{frame} for $\mathcal{H}$ if there exist constants $0<A\leq B < \infty$ such that
\begin{equation} \label{eqn:1}
\forall f \in \mathcal{H}\textrm{, } \qquad A\Vert f \Vert ^2 \leq \sum_{j \in J} \vert \langle f,e_{j} \rangle \vert ^2 \leq B\Vert f \Vert ^2 {.}
\end{equation}
The maximal such $A$ and minimal such $B$ are the \emph{optimal frame bounds}. In this paper, the phrase \emph{frame bound} will always mean the optimal frame bound, where $A$ is the \emph{lower frame bound} and $B$ is the \emph{upper frame bound}.  A frame is \emph{Parseval} if $A = B = 1$. 
\end{defn}
Every orthonormal basis is a frame. One may view frames as generalizations of orthonormal bases which mimic the reconstruction properties (i.e.: $\forall x, x= \sum \langle x, e_j \rangle e_j$)  of orthonormal bases but may have some redundancy.   Frames first appeared in the seminal paper by Duffin and Schaeffer \cite{DS52}.
\begin{defn} 
Let $\psi \in L^2 \left( \mathbb{R}^d \right)$ and define the \emph{(dyadic) wavelet system},
\begin{equation*}
\mathcal{W} \left( \psi \right) = \lbrace D_n T_k \psi(x) : n \in \mathbb{Z}, k \in \mathbb{Z}^d \rbrace = \lbrace 2^{nd/2} \psi\left( 2^n x - k \right) : n \in \mathbb{Z}, k \in \mathbb{Z}^d \rbrace \mathrm{.}
\end{equation*}
If $\mathcal{W} \left( \psi \right)$ is an orthonormal basis (respectively, Parseval frame, frame) for $L^2 \left( \mathbb{R}^d \right)$, then $\psi$ is an \emph{orthonormal wavelet} (respectively, \emph{Parseval frame wavelet, frame wavelet}) or simply a \emph{wavelet} for $L^2(\mathbb{R}^d)$.
\end{defn}
The first wavelet system appeared in Haar's thesis \cite{Haar09} \cite{Haar10}.  For any measurable set $S \subseteq \mathbb{R}^d$, the \emph{characteristic function of $S$}, $\mathbbm{1}_S$, is
\begin{equation*}
\mathbbm{1}_S(x) = \left\{ \begin{array}{ccc} 1 & ; & x\in S \\ 0 &  ; & \mathrm{else} \end{array} \right..
\end{equation*}
\begin{defn}
If $L$ is a measurable subset of $\widehat{\mathbb{R}}^d$ and $\mathcal{W} ( \check{\mathbbm{1}}_L )$ is an orthonormal basis (respectively, frame or Parseval frame) for $L^2(\mathbb{R}^d)$, then $L$ is an \emph{orthogonal} (respectively, \emph{frame} or \emph{Parseval frame}) \emph{wavelet set} or simply a \emph{wavelet set}.
\end{defn}
Wavelet sets are completely characterized by simple geometric properties.  Classical examples of such wavelets which predate formal wavelet set theory are the \emph{Shannon} or \emph{Littlewood-Paley wavelet} $\check{\mathbbm{1}}_{[-1,-1/2)\cup[1/2,1)}$ and the \emph{Journ\'{e} wavelet} $\check{\mathbbm{1}}_{[-16/7,-2)\cup[-1/2,-2/7)\cup[2/7,1/2)\cup[2,16/7)}$.  Finally, we comment on the conventions used in this paper.  For a function $f \in L^1(\mathbb{R}^d)$, the \emph{Fourier transform of $f$} is defined to be
\begin{equation*}
\mathcal{F}(f)(\gamma) = \hat{f}(\gamma) = \int f(x)e^{-2 \pi i x \cdot \gamma} dx \textrm{.}
\end{equation*}
By Plancherel's Theorem, $\mathcal{F}$ extends from $L^1 \cap L^2$ to a unitary operator $L^2 \rightarrow L^2$. We denote the inverse Fourier transform of a function $g \in L^2(\widehat{\mathbb{R}}^d)$ as $\mathcal{F}^{-1} g = \check{g}$. Our definition of support will not be the traditional one to get around conflicts between measure and topological closure.  For $f: \mathbb{R}^d \rightarrow \mathbb{C}$, the \emph{support of $f$}, $\supp f$ is the following equivalence class of measurable sets
\begin{equation*}
\left\{ S \subseteq \mathbb{R}^d: \int_{\mathbb{R}^d\backslash S} |f(x)| dx=0, \textrm{ and if $R \subset S$ then $\int_{ S \backslash R} |f(x)| dx \geq 0$ } \right\} .
\end{equation*}
We shall still speak of \emph{the} support of a function, just as we refer to \emph{a} function in an $L^p$ space.  So, $\supp f \subseteq \mathbb{R}^d$ means that at least one element in the equivalence class is a subset of $\mathbb{R}^d$, and $f$ \emph{is compactly supported} means that $\supp f \subseteq K$, where $K$ is a compact set.  Similarly, if $f:\mathbb{R}^d \rightarrow \mathbb{C}$ and for $\epsilon > 0$, define $\supp_\epsilon f$ to be the the equivalence class of measurable sets.
\begin{equation*}
\left\{ S \subseteq \mathbb{R}^d: \int_{\mathbb{R}^d\backslash S}( |f(x)| - \epsilon) dx <0, R \subset S \Rightarrow \int_{ S\backslash R} (|f(x)|-\epsilon) dx \geq  0 \right\} .
\end{equation*}
%%%%%%%%%%%%%%%%%%%%%%%%%%%%%
\subsection{Background}
%%%%%%%%%%%%%%%%%%%%%%%%%%%%%%
\begin{defn}
The space $C_c^\infty(\mathbb{R}^d)$ consists of functions $f: \mathbb{R}^d \rightarrow \mathbb{C}$ which are infinitely differentiable and compactly supported.  Given a multi-index $\alpha = (\alpha_1, \alpha_2, \hdots, \alpha_d) \in \left(\mathbb{N}\cup\{0\}\right)^d$, we write $\vert \alpha \vert = \sum_{i=1}^d \alpha_i$, $x^\alpha = \prod_{i=1}^d x_i^{\alpha_i}$, and $D^\alpha = \frac{\partial_{\alpha_1}}{\partial x_1^{\alpha_1}}\frac{\partial_{\alpha_2}}{\partial x_2^{\alpha_2}} \cdots \frac{\partial_{\alpha_d}}{\partial x_d^{\alpha_d}}$. An infinitely differentiable function $f : \mathbb{R}^d \rightarrow \mathbb{C}$ is an element of the \emph{Schwartz space} $\mathscr{S}(\mathbb{R}^d)$ if
\begin{equation*}
\forall n = 0,1,\hdots \qquad \sup_{\vert \alpha \vert \leq n, \alpha \in \left(\mathbb{N}\cup\{0\}\right)^d} \sup_{x \in \mathbb{R}^d} \left( 1 + \Vert x \Vert^2 \right)^\alpha \vert D^\alpha f (x)\vert < \infty.
\end{equation*}
Clearly $\mathscr{S} \subseteq L^1$, so the Fourier transform is well defined on $\mathscr{S}$ and is in fact a topological automorphism. Since $C_c^\infty \subseteq \mathscr{S}$, the (inverse) Fourier transform of a smooth compactly supported function is smooth.  We will denote the Hardy space $\lbrace f\in L^2(\mathbb{R}) : \supp \hat{f} \subseteq [0, \infty) \rbrace$ as $H^2(\mathbb{R})$, as in \cite{BHn97}.
\end{defn}
We now make note of a comprehensive result concerning frame bound estimation, which appeared in \cite{BenKing09} based on results in \cite{Dau92}, \cite{Chr02}, \cite{CS93b}, and \cite{Jing99} and may be viewed as a specific application of a result in \cite{RS97}.
\begin{thm}\label{thm:dau}\label{thm:par}
Let $\psi \in L^2( \mathbb{R}^d )$, and let $a>0$ be arbitrary. Define
\begin{eqnarray*}
M_{\psi} & = &  \operatorname{ess sup}_{a \leq \Vert \gamma \Vert \leq 2a} \sum_{k \in \mathbb{Z}^d} \sum_{n \in \mathbb{Z}} \left\vert \hat{\psi} \left(2^n \gamma \right) \hat{\psi} \left(2^n \gamma + k \right) \right\vert \textrm{ and}\\
N_{\psi} & = & \operatorname{ess inf}_{a \leq \Vert \gamma \Vert \leq 2a} \left[ \sum_{n \in \mathbb{Z}} \left\vert \hat{\psi} \left(2^n \gamma \right) \right\vert ^2 - \sum_{k \neq 0} \sum_{n \in \mathbb{Z}} \left\vert \hat{\psi} \left(2^n \gamma \right) \hat{\psi} \left(2^n \gamma + k \right) \right\vert \right] \mathrm{.}
\end{eqnarray*}
If $M_{\psi} < \infty$ and $N_{\psi} > 0$, then $\mathcal{W}\left(\psi\right)$ is a frame with frame bounds $A$ and $B$ satisfying the inequality $N_{\psi} \leq A \leq B \leq M_{\psi}$. Furthermore, if $\mathcal{W}(\psi)$ is a frame, then for 
\begin{eqnarray*}
\kappa_\psi (\gamma) & = & \sum_{n \in \mathbb{Z}} \left\vert \hat{\psi} \left( 2^n \gamma \right) \right\vert^2, \\
\overline{K}_\psi & = & \operatorname{ess sup}_{a \leq \Vert \gamma \Vert \leq 2a} \kappa_\psi (\gamma) \textrm{, and} \\
\underline{K}_\psi & = & \operatorname{ess inf}_{a \leq \Vert \gamma \Vert \leq 2a} \kappa_\psi (\gamma),
\end{eqnarray*}
the inequality $A \leq \underline{K}_\psi \leq \overline{K}_\psi \leq B$ holds.
\end{thm}
Note that this Theorem implies that $\mathcal{W}(\psi)$ is a Parseval frame only if $\kappa_\psi(\gamma) = 1$ for almost all $\gamma$.
%%%%%%%%%%%%%%%%%%%%%%%%%%
\subsection{Outline and Results}
%%%%%%%%%%%%%%%%%%%%%%%%%%
In Section~\ref{sec:han1}, we present the results from \cite{BHn94} and \cite{BHn97} which concern the existence of smooth Parseval frames which approximate $1$-dimensional Parseval frame wavelet sets.  Bin Han's methods involve auxiliary smooth functions which we try to generalize to higher dimensions in Section~\ref{sec:fail}.  We show that forming tensor products or other similarly modified versions of the auxiliary functions from Section~\ref{sec:han1} either fails to yield a Parseval frame or fails to yield a smooth wavelet when used to smooth a certain type of wavelet set.  However, some Parseval wavelet set wavelets in $\widehat{\mathbb{R}}^d$ can be smoothed using methods inspired by Han's work, see Section \ref{sec:construct}.  In Section~\ref{sec:parti} we construct a class of $C_c^\infty$ functions which form frames with upper frame bounds converging to $1$.   We conclude with a summary in Section~\ref{section:con} and lengthy calculations in Section~\ref{section:app}.

%%%%%%%%%%%%%%%%%%%%%
\section{Schwartz class Parseval frames}
%%%%%%%%%%%%%%%%%%%%%%%%%%%%
\subsection{Parseval frames for $L^2(\mathbb{R})$} \label{sec:han1}
%%%%%%%%%%%%%%%%%%%%%%%%%%%%
In his Master's thesis, \cite{BHn94}, as well as the paper \cite{BHn97}, Bin Han proved the existence of $C^\infty$ Parseval frames for $H^2(\mathbb{R})$. The following definition and two lemmas appear in the paper \cite{BHn97}.
\begin{defn}
For any measurable subset $L \subseteq \widehat{\mathbb{R}}^d$ define
\begin{equation*}
\Delta(L) = \operatorname{dist}\big(L, \bigcup_{k \in \mathbb{Z}^d \backslash \{0\}}(L + k)\big).
\end{equation*}
\end{defn} 
\begin{lem}\label{lem:theta}
There exists a function $\theta \in C^\infty(\mathbb{R})$ satisfying $\theta(x) = 0$ when $x \leq -1$ and $\theta(x) = 1$ when $x \geq 1$ and
\begin{equation*}
\theta(x)^2 + \theta(-x)^2 = 1, \qquad x \in \mathbb{R}.
\end{equation*}
\end{lem}
\begin{defn} \label{defn:han}
Given a closed interval $I = [a,b]$ and two positive numbers $\delta_1$, $\delta_2$ such that $\delta_1 + \delta_2 \leq b-a$, we define
\begin{equation*}
f_{(I;\delta_1,\delta_2)}(x) = \left\{ \begin{array}{lr}
\theta\left(\frac{x-a}{\delta_1}\right) & \textrm{when $x < a + \delta_1$} \\
1 & \textrm{when $a + \delta_1 \leq x \leq b - \delta_2$} \\
\theta\left(\frac{b-x}{\delta_2}\right) & \textrm{when $x > b - \delta_2$}
\end{array}\right.
\end{equation*}
Note that $\supp(f_{(I;\delta_1,\delta_2)}) \subseteq [a - \delta_1, b + \delta_2]$.
\end{defn}
\begin{lem}\label{lem:han}
For any positive numbers $\delta_1$, $\delta_2$, $\delta_3$ and $0 < a < b< c$,
\begin{equation*}
f_{(I;\delta_1,\delta_2)}(2^{n}x) = f_{(2^{-k}I;2^{-k}\delta_1,2^{-k}\delta_2)}(x)
\end{equation*}
and
\begin{equation*}
f^2_{([a,b];\delta_1,\delta_2)}(x) + f^2_{([b,c];\delta_2,\delta_3)}(x) = f^2_{([a,c];\delta_1,\delta_3)}(x).
\end{equation*}
\end{lem}
The preceding lemmas are used to prove
\begin{prop}[\cite{BHn97}] 
Suppose that a family of disjoint closed intervals $I_i = [a_i, b_i]$, $1 \leq i \leq l$ in $(0, \infty)$ is arranged in a decreasing order, i.e., $0 < b_l < b_{l-1} < \ldots < b_1$ and $\cup_{i=1}^l I_i$ is a Parseval frame wavelet set for $H^2(\mathbb{R})$. If $\Delta(\cup_{i=1}^l I_i) > 0$.  Then for any
\begin{equation*}
0 < \delta < \frac{1}{2}\min\lbrace \Delta(\cup_{i=1}^l I_i), \min_{1 \leq i \leq l} \lbrace b_i - a_i \rbrace, \min_{1 \leq i < l} \operatorname{dist}(I_i, I_{i+1}) \rbrace,
\end{equation*}
let
\begin{equation*}
\hat{\psi}_\delta = f_{(I_1;\frac{\delta}{2},\delta)} + \sum_{i=2}^l f_{(I_i;2^{-k_i - 1}\delta, 2^{-k_i - 1}\delta)}
\end{equation*}
where $k_i$ is the unique non-negative integer such that $2^{k_i}I_i \subseteq [\frac{1}{2}b_1,b_1]$. We have $\psi_\delta \in \mathscr{S}(\mathbb{R})$ and $\mathcal{W}(\psi_\delta)$ is a Parseval frame in $H^2(\mathbb{R})$.
\end{prop}
A similar proposition holds for $L^2(\mathbb{R})$.  A proof may be found in \cite{Dissert}.
\begin{prop}  \label{prop:BH}
Suppose that a family of disjoint closed intervals $I_i = [a_i, b_i]$, $1 \leq i \leq l$ in $\widehat{\mathbb{R}}$ is arranged in a decreasing order, i.e., $b_l < b_{l-1} < \ldots < b_1$ where $b_j < 0 < a_{j-1}$ and $\cup_{i=1}^l I_i$ is a Parseval frame wavelet set for $L^2(\mathbb{R})$. If $\Delta(\cup_{i=1}^l I_i) > 0$, then for any
\begin{equation*}
0 < \delta < \frac{1}{2}\min\lbrace \Delta(\cup_{i=1}^l I_i), \min_{1 \leq i \leq l} \lbrace b_i - a_i \rbrace, \min_{1 \leq i < l} \operatorname{dist}(I_i, I_{i+1}) \rbrace,
\end{equation*}
let
\begin{equation*}
\hat{\psi}_\delta = f_{(I_1;\frac{\delta}{2},\delta)} + \left[\sum_{i=2}^{l-1} f_{(I_i;2^{-k_i - 1}\delta, 2^{-k_i - 1}\delta)}\right] + f_{(I_l; \delta, \frac{\delta}{2})}
\end{equation*}
where for $2 \leq i \leq j-1$, $k_i$ is the unique non-negative integer such that $2^{k_i}I_i \subseteq [\frac{1}{2}b_1,b_1]$ and for $j \leq i \leq l-1$, $k_i$ is the unique non-negative integer such that $2^{k_i}I_i \subseteq [a_l, \frac{1}{2}a_l]$. We have $\psi_\delta \in \mathscr{S}(\mathbb{R})$ and $\mathcal{W}(\psi_\delta)$ is a Parseval frame in $L^2(\mathbb{R})$.
\end{prop}
The following theorem from \cite{BenKing09} was used to show that shrinking the frequency support of smoothed Parseval frame wavelets (which were not always even frames, let along Parseval frames) tightens the frame bounds.  The shrinking of the frequency support is related to increasing the sampling of the continuous wavelet system.
\begin{thm} \label{thm:shrink}
Let $\hat{\psi} \in L_c^\infty (\widehat{\mathbb{R}}^d)$ be a non-negative function.  If there exists an $\epsilon > 0$ such that for $L = \supp_\epsilon \hat{\psi}$, $\bigcup_{n\in\mathbb{Z}} 2^n L = \widehat{\mathbb{R}}^d$ up to a set of measure $0$, and for $\tilde{L} = \supp \hat{\psi}$, $\Delta(\tilde{L}) > 0$, and $\operatorname{dist} (0,\tilde{L}) >0$. Then, $\mathcal{W}(\psi)$ is a frame for $L^2(\mathbb{R}^d)$. The frame bounds are $\operatorname{ess inf}_\gamma \kappa_\psi (\gamma)$ and $\operatorname{ess sup}_\gamma \kappa_\psi (\gamma)$.
\end{thm}
Using this theorem, we may modify any bounded Parseval frame wavelet set in $\widehat{\mathbb{R}}$ so that we may apply Proposition \ref{prop:BH} in order to obtain a smooth Parseval frame wavelet set.
\begin{cor}
Let $L \subseteq \widehat{\mathbb{R}}$ be a Parseval frame wavelet set. Let $N \in \mathbb{Z}$ have the trait that $\overline{2^N L} \subseteq (-\frac{1}{2}, \frac{1}{2} )$. Then there exists a $\psi \in \mathscr{S}(\mathbb{R})$ such that $\hat{\psi} \in C_c^\infty(\mathbb{R})$ and $\mathcal{W}(\psi)$ is a Parseval frame and the measure of $\supp (\psi) \backslash 2^N L$ is arbitrarily small. 
\end{cor}
%%%%%%%%%%%%%%%%%%%%%%%%%%%%%
\subsection{Extensions of Han's construction}\label{sec:fail}
%%%%%%%%%%%%%%%%%%%%%%%%%%%%%
We would like to extend Han's results even further in order to create Schwartz class Parseval frames over $L^2(\mathbb{R}^d)$ for $d>1$. The basic idea of Han's construction is to replace each $\mathbbm{1}_{[a_i, b_i]}(x)$ with an appropriate $C_c^\infty$ bump function $f_{([a_i,b_i];\delta_i,\tilde{\delta}_i)}(x)$, where $\cup_i [a_i, b_i]$ is a Parseval frame wavelet set with $\Delta (\cup_i [a_i,b_i])>0$. We will attempt to generalize the smoothing techniques on the class of Parseval frame wavelet sets
\begin{equation*}
\{L_a\} = \{ [-2a, 2a]^2\backslash[-a,a]^2:0 < a < \frac{1}{4} \}.
\end{equation*}
Any such $L_a$ is indeed a Parseval frame wavelet set since each tiles the plane under dyadic dilation and $\Delta(L_a)>0$ for $0<a<\frac{1}{4}$ \cite{BSu06} \cite{DLS98}. These sets are natural ones to start with because of their simplicity. We need to define an appropriate family of smooth functions to replace each $\mathbbm{1}_{L_a}(x,y)$. We try the following functions:
\begin{eqnarray}
&& h_{(L_a;\delta,\frac{\delta}{2})}(x,y)\\\nonumber
&&  = f_{([-2a,2a];\delta,\delta)}(x)f_{([-2a,2a];\delta,\delta)}(y) - f_{([-a,a];\frac{\delta}{2},\frac{\delta}{2})}(x)f_{([-a,a];\frac{\delta}{2},\frac{\delta}{2})}(y), \textrm{and} \label{eqn:h}
\end{eqnarray}
\begin{eqnarray}
&& g_{(L_a;\delta,\frac{\delta}{2})}(x,y) \label{eqn:g}\\
&&  = \left\{ \begin{array}{lr}
\theta\left(\frac{2a-|y|}{\delta}\right) &\textrm{when $|x| \in [0,2a-\delta]$ and $|y| \in [2a - \delta, 2a + \delta]$} \\
\theta\left(\frac{2a-|y|}{\delta}\right)\theta\left(\frac{2a-|x|}{\delta}\right) & \textrm{when $|x|,|y| \in [2a - \delta, 2a + \delta]$} \\
\theta\left(\frac{2a-|x|}{\delta}\right) & \textrm{when $|y| \in [0,2a-\delta]$ and $|x| \in [2a - \delta, 2a + \delta]$} \\
1 & \textrm{when $(|x|,|y|)^T \in [0,2a-\delta]^2 \backslash [0,a + \frac{\delta}{2}]^2$} \\
\theta\left(\frac{|y|-a}{\delta/2}\right) & \textrm{when $|x| \in [0,a-\frac{\delta}{2}]$ and $|y| \in [a - \frac{\delta}{2}, a + \frac{\delta}{2}]$} \\
\theta\left(\frac{|y|-a}{\delta/2}\right)\theta\left(\frac{|x|-a}{\delta/2}\right) & \textrm{when $|x|,|y| \in [a - \frac{\delta}{2}, a + \frac{\delta}{2}]$} \\
\theta\left(\frac{|x|-a}{\delta/2}\right) & \textrm{when $|y| \in [0,a-\frac{\delta}{2}]$ and $|x| \in [a - \frac{\delta}{2}, a + \frac{\delta}{2}]$} \\
0& \textrm{otherwise},\nonumber
\end{array} \right.
\end{eqnarray}
where $\theta$ is as in Lemma \ref{lem:theta}. We first note that $g$ is well defined even though the piecewise domains overlap. In order to form $h$, we tensor the $1$-dimensional interval bump functions to create $2$-dimensional rectangle bump functions and then subtract such functions corresponding to $[-2a,2a]^2$ and $[-a,a]^2$. The function $g$ may be seen as a piecewise tensor product. In fact, $h_{(L_a;\delta,\frac{\delta}{2})}(x,y) = g_{(L_a;\delta,\frac{\delta}{2})}(x,y) $
\begin{equation*}
\textrm{ for $(x,y)^T \notin [-a -\frac{\delta}{2}, a +\frac{\delta}{2}]^2\backslash [-a +\frac{\delta}{2}, a -\frac{\delta}{2}]^2$}
\end{equation*}
and $\supp h_{(L_a;\delta,\frac{\delta}{2})} = \supp g_{(L_a;\delta,\frac{\delta}{2})}$. Although both of these functions seem like promising candidates, neither $\sum_{n\in\mathbb{Z}} h_{(L_a;\delta,\frac{\delta}{2})}^2(2^n x, 2^n y)$ nor $\sum_{n\in\mathbb{Z}} g_{(L_a;\delta,\frac{\delta}{2})}^2(2^n x, 2^n y)$ are equal to $1$ almost everywhere. It follows from Theorem \ref{thm:par} that neither $\mathcal{W}(\check{h})$ nor $\mathcal{W}(\check{g})$ are Parseval frames.  The proofs of the following 3 propositions are in Section~\ref{section:app}.
\begin{prop}\label{prop:h}
Let $0 < a < \frac{1}{4}$, set $L = [-2a,2a]^2 \backslash [-a,a]^2$, and pick a $\delta$ such that $0 < \delta < \frac{1}{2} \min \{1-4a,a \}$. Let $h$ be as in Equation \ref{eqn:h}. Then $\sum_{n\in\mathbb{Z}} h_{(L_a;\delta,\frac{\delta}{2})}^2(2^n x, 2^n y)$ is not equal to $1$ a.e.
\end{prop}
\begin{prop}\label{prop:g}
Let $0 < a < \frac{1}{4}$, set $L = [-2a,2a]^2 \backslash [-a,a]^2$, and pick a $\delta$ such that $0 < \delta < \frac{1}{2} \min \{1-4a,a \}$. Let $g$ be as in Equation \ref{eqn:g}. Then $\sum_{n\in\mathbb{Z}} g_{(L_a;\delta,\frac{\delta}{2})}^2(2^n x, 2^n y)$ is not equal to $1$ a.e.
\end{prop}
However, as the calculations in Section~\ref{section:app} show, 
\begin{equation*}
\sum_{n \in \mathbb{Z}} g^2_{(L_a;\delta,\frac{\delta}{2})} (2^mx,2^ny)=1 \textrm{ for all}
\end{equation*}
\begin{equation*}
(|x|,|y|)^T \notin \{0\} \cup \left(
\bigcup_{m \in \mathbb{Z}} 2^m \left[a - \frac{\delta}{2}, a + \frac{\delta}{2}\right]^2 \right).
\end{equation*}
So we adjust $h_{(L_a;\delta,\frac{\delta}{2})}(x,y)$ on $C = [a - \frac{\delta}{2}, a + \frac{\delta}{2}]^2 \cup [2a - \delta, 2a + \delta]^2$ in hopes of obtaining a Parseval frame. We do this by setting
\begin{equation*}
f_{(L_a;\delta,\frac{\delta}{2})}(x,y) = h_{(L_a;\delta,\frac{\delta}{2})}(x,y) \textrm{ for $(|x|,|y|)^T \notin C$}
\end{equation*}
and
\begin{equation*}
f_{(L_a;\delta,\frac{\delta}{2})}(x,y) = f_{(L_a;\delta,\frac{\delta}{2})}(\tilde{x},\tilde{y}),
\end{equation*}
for $(|x|,|y|)^T,(|\tilde{x}|,|\tilde{y}|)^T \in C$, $|x|+|y|=|\tilde{x}|+|\tilde{y}|$, and $|x|+|y|$ small enough.  Explicitly, $f_{(L_a;\delta,\frac{\delta}{2})}(x,y) =$
\begin{equation*}
\left\{ \begin{array}{lr}
\theta\left(\frac{2a-|y|}{\delta}\right) & \textrm{$|x| \in [0, 2a-\delta]$ and $|y| \in [2a-\delta,2a+\delta]$} \\
\theta\left(\frac{4a-|x|-|y|-\delta}{\delta} \right) & \textrm{$|x|,|y| \in [2a - \delta, 2a + \delta]$, for $4a-2\delta \leq |x|+|y| \leq 4a$}\\
\theta\left(\frac{2a-|x|}{\delta}\right) & \textrm{$|y| \in [0, 2a-\delta]$ and $|x| \in [2a-\delta,2a+\delta]$} \\
1 & \textrm{$(|x|,|y|)^T \in [0,2a-\delta]^2 \backslash [0,a + \frac{\delta}{2}]^2$} \\
1 & \textrm{ $|x|,|y| \in [a - \frac{\delta}{2},a +\frac{\delta}{2}]$ for $2a\leq |x|+|y| \leq 2a + \delta$} \\
\theta\left(\frac{|y|-a}{\delta/2}\right) & \textrm{$|x| \in [0, a-\frac{\delta}{2}]$ and $|y| \in [a-\frac{\delta}{2},a+\frac{\delta}{2}]$} \\
\theta\left(\frac{|x|+|y|-2a +\delta/2}{\delta/2} \right) & \textrm{$|x|,|y| \in [a - \frac{\delta}{2},a +\frac{\delta}{2}]$ for $2a -\delta\leq |x|+|y| \leq 2a$} \\
\theta\left(\frac{|x|-a}{\delta/2}\right) & \textrm{$|y| \in [0, a-\frac{\delta}{2}]$ and $|x| \in [a-\frac{\delta}{2},a+\frac{\delta}{2}]$} \\
0 &  \textrm{otherwise.}
\end{array} \right.
\end{equation*}
\begin{prop}\label{prop:f}
Let $0 < a < \frac{1}{4}$, set $L = [-2a,2a]^2 \backslash [-a,a]^2$, and pick a $0 < \delta < \frac{1}{2} \min \{1-4a,a \}$. Let $\hat{\psi}_\delta = f_{(L_a;\delta,\frac{\delta}{2})}$. For $x, y >0$,
\begin{equation*}
\hat{\psi}_\delta^2(\vec{x}) + \hat{\psi}_\delta^2(2\vec{x}) = f_{([-2a,2a]^2\backslash[-\frac{a}{2},\frac{a}{2}]^2;\delta,\frac{\delta}{4})}^2(\vec{x}).
\end{equation*}
\end{prop}
\begin{prop}
Let $0 < a < \frac{1}{4}$, set $L = [-2a,2a]^2 \backslash [-a,a]^2$, and pick a $0 < \delta < \frac{1}{2} \min \{1-4a,a \}$. Let $\hat{\psi}_\delta = f_{(L_a;\delta,\frac{\delta}{2})}$. Then $\hat{\psi}_\delta \in C_c(\widehat{\mathbb{R}}^2)\backslash C_c^1(\widehat{\mathbb{R}}^2)$ and $\mathcal{W}(\psi_\delta)$ is a Parseval frame for $L^2(\mathbb{R}^2)$.
\end{prop}
\begin{proof}
Since
\begin{eqnarray*}
\supp \hat{\psi}_\delta & \subseteq & [-2a-\delta,2a+\delta]^2\\
& \subseteq & \left(-2a-\frac{1}{2}(1-4a),2a+\frac{1}{2}(1-4a)\right)^2 =\left(-\frac{1}{2},\frac{1}{2}\right)^2,
\end{eqnarray*}
for all $n \in \mathbb{N}\cup\{0\}$ and $k \in \mathbb{Z}^2 \backslash \{0\}$,
\begin{equation*}
\hat{\psi}_\delta \left(2^n(\vec{x} + k)\right)\overline{\hat{\psi}_\delta(2^n\vec{x})}=0 \textrm{ a.e.},
\end{equation*}
where $\vec{x} = (x,y)$. In order to utilize Theorem \ref{thm:par}, we would like to show that $\sum_{n \in \mathbb{Z}} \left\vert \hat{\psi}_\delta (2^n\vec{x}) \right\vert^2 =1$  for a.e.~$\vec{x}$.  We will accomplish this by showing that $\hat{\psi}_\delta^2(\vec{x}) + \hat{\psi}_\delta^2(2\vec{x}) = f_{([-2a,2a]^2\backslash[-\frac{a}{2},\frac{a}{2}]^2;\delta,\frac{\delta}{4})}^2(\vec{x}).$  Then it will follow from iteration that
\begin{equation*}
\sum_{n=M}^{N} \left\vert \hat{\psi}_\delta (2^n \vec{x})\right\vert^2 = f^2_{([-2^{1-M}a,2^{1-M}a]^2\backslash[-2^{-N}a,2^{-N}a]^2;2^{-M}\delta,2^{-1-N}\delta)}(\vec{x}),
\end{equation*}
which is $1$ on $[2^{1-M}a+2^{-M}\delta, 2^{1-M}a-2^{-M}\delta]^2\backslash [2^{-N}a+2^{-1-N}\delta, 2^{-N}a-2^{-1-N}\delta]^2$, where
\begin{equation*}
2^{1-M}a-2^{-M}\delta = 2^{-M}(2a-\delta) \rightarrow \infty \quad \textrm{as } M \rightarrow -\infty
\end{equation*}
and
\begin{equation*}
2^{-N}a-2^{-1-N}\delta = 2^{-N}(a-\frac{\delta}{2}) \rightarrow 0 \quad \textrm{as } N \rightarrow \infty
\end{equation*}
So $\sum_{n\in\mathbb{Z}} \left\vert \hat{\psi}_\delta (2^n \vec{x})\right\vert^2 =1 \textrm{ a.e.}$ By symmetry, it will suffice to show that 
\begin{equation*}
\hat{\psi}_\delta^2(\vec{x}) + \hat{\psi}_\delta^2(2\vec{x}) = f_{([-2a,2a]^2\backslash[-\frac{a}{2},\frac{a}{2}]^2;\delta,\frac{\delta}{4})}^2(\vec{x})
\end{equation*}
for positive $x$ and $y$, but this follows immediately from the preceding proposition.

Thus $\mathcal{W}(\psi_\delta)$ is a Parseval frame, but $\hat{\psi}_\delta$ has cusps along $\{(2a-\delta+t,2a-\delta)^T:0\leq t \leq 2\delta\}$, as well as other edges. So $\hat{\psi}_\delta \notin C_c^1(\widehat{\mathbb{R}}^2)$.
\end{proof}

Thus we have found a method to smooth the Parseval frame wavelets $\check{\mathbbm{1}}_{L_a}$ for $0 < a < \frac{1}{4}$, which is analogous to Han's method, but it does not yield Parseval frame wavelets with good temporal decay like Schwartz functions. It seems that this method should generalize to other Parseval frame wavelet sets in $\widehat{\mathbb{R}}^2$ which have piecewise horizontal and vertical boundaries. However, there does not seem to be an easy way to write an explicit formula that works in general. Furthermore, only a relatively small number of Parseval frame wavelet sets have such a boundary. Perhaps not all is lost. Instead of trying to smooth $\mathbbm{1}_K$ for some already chosen Parseval frame wavelet set $K$, we now try to build Schwartz class Parseval frames for $L^2(\mathbb{R}^2)$ directly from the $C_c^\infty$ bump functions over $\widehat{\mathbb{R}}$. 
%%%%%%%%%%%%%%%%%%%%%%%%%
\section{A construction in higher dimensions}\label{sec:construct}
%%%%%%%%%%%%%%%%%%%%%%%%%%
In the preceding work, problems arose around the corners of the boundary of $L_a =[-2a,2a]^2 \backslash [-a,a]^2$ when we tried to smooth $\mathbbm{1}_{L_a}$. What if there were no corners to deal with? For $0 < a <\frac{1}{4}$, we define
\begin{equation*}
f_{([a,2a]\times S^1;\frac{\delta}{2},\delta)}(x,y) = f_{([a,2a];\frac{\delta}{2},\delta)}(\sqrt{x^2 + y^2}),
\end{equation*}
where $f_{([a,2a];\frac{\delta}{2},\delta)}(\cdot)$ is as in Definition \ref{defn:han}.
\begin{prop}
Let $0 < a < \frac{1}{4}$. For any $0 < \delta< \frac{1}{2}\min\{1-4a,a\}$, define $\hat{\psi}_\delta : \mathbb{R}^2 \rightarrow \mathbb{R}$ by $\hat{\psi}_\delta(x,y) = f_{([a,2a];\frac{\delta}{2},\delta)}(\sqrt{x^2 + y^2})$. Then, $\hat{\psi}\in\mathscr{S}(\mathbb{R}^2)$ and $\mathcal{W}(\psi_\delta)$ is a Parseval frame for $L^2(\mathbb{R}^2)$.
\end{prop}
\begin{proof}
By construction, $\hat{\psi}_\delta \in C_c^\infty(\widehat{\mathbb{R}}^2) \Rightarrow \psi_\delta \in \mathscr{S}(\mathbb{R}^2)$. Since $\delta < \frac{1}{2}(1 -4a)$, $\Delta(\supp \hat{\psi}_\delta) <0$. So for all $n \in \mathbb{N}\cup\{0\}$ and $k \in \mathbb{Z} \backslash \{0\}$,
\begin{equation*}
\hat{\psi}_\delta(2^n(\vec{x}+k))\overline{\hat{\psi}(2^n\vec{x}}) = 0 \textrm{ a.e.}
\end{equation*}
where $\vec{x} = (x,y)^T$. Hence, in order to prove that $\mathcal{W}(\psi_\delta)$ is a Parseval frame, it suffices to show that $\sum_{n\in\mathbb{Z}}\left\vert \hat{\psi}_\delta (2^n \vec{x}) \right\vert^2 = 1$ a.e. We compute
\begin{eqnarray}
\sum_{n\in\mathbb{Z}} \left\vert \hat{\psi} (2^n \vec{x}) \right\vert^2 & = & \sum_{n\in\mathbb{Z}} \left\vert f_{([a,2a];\frac{\delta}{2},\delta)} \left( \sqrt{(2^nx)^2 + (2^ny)^2} \right) \right\vert^2 \nonumber \\
& = & \sum_{n\in\mathbb{Z}} \left\vert f_{([a,2a];\frac{\delta}{2},\delta)} (2^n z) \right\vert^2 \textrm{ for $z=\sqrt{x^2+y^2}$}  \label{eqn:2}
\end{eqnarray}
We know that (\ref{eqn:2}) $=1$ for almost all non-negative $z$, specifically for $z>0$. So $\sum_{n\in\mathbb{Z}} \left\vert \hat{\psi}_\delta (2^n \hat{x})\right\vert^2 =1$ for $\widehat{\mathbb{R}}^2 \ni \vec{x} \neq 0$. Thus $\mathcal{W}(\psi_\delta)$ is a Parseval frame for $L^2(\mathbb{R}^2)$.
\end{proof}
This result and proof generalize to $\mathbb{R}^d$, $d >2$.
\begin{cor}
Let $0 < a < \frac{1}{4}$. For any $0 < \delta <\frac{1}{2}\min\{1-4a,a\}$, define $\hat{\psi}_\delta : \mathbb{R}^d \rightarrow \mathbb{R}$ by $\hat{\psi}_\delta(x) = f_{([a,2a];\frac{\delta}{2},\delta)}(\Vert x \Vert)$. Then, $\hat{\psi}\in\mathscr{S}(\mathbb{R}^d)$ and $\mathcal{W}(\psi_\delta)$ is a Parseval frame for $L^2(\mathbb{R}^d)$.
\end{cor}
\begin{proof}
The proof is as above.
\end{proof}
We now have Schwartz class Parseval frames for $L^2(\mathbb{R}^d)$, $d>1$ which are elementary to describe.
%%%%%%%%%%%%%%%%%%%%%%%%%
\section{Partitions of unity}\label{sec:parti}
%%%%%%%%%%%%%%%%%%%%%%%%%%
We now switch gears slightly and consider a class of well-known functions in $\mathcal{S}(\mathbb{R})$.  $C^\infty$ partitions of unity are important tools in analysis and differential topology.  While the topic of $C^\infty$ partitions of unity is outside the scope of this paper, we shall utilize a class of functions which is commonly used in conjunction with that subject, e.g.: \cite{Lee}. 
\begin{defn}
Let $f: \widehat{\mathbb{R}} \rightarrow \mathbb{R}$ be the function $f(\gamma) = e^{-\frac{1}{\gamma}} \mathbbm{1}_{(0,\infty)}$.  Also, let $b,m>0$ be such that $b - \frac{1}{m} > 0$. Define $\hat{\varphi} \in C_c^\infty(\widehat{\mathbb{R}})$ as
\begin{equation*}
\hat{\varphi}_{b,m} (\gamma) = \frac{f(b + \frac{1}{m} - |\gamma|)}{f(b + \frac{1}{m} - |\gamma|)+f(|\gamma| - b + \frac{1}{m})}.
\end{equation*}
If $a > \frac{4}{m}$ is clear from the context, we shall write $\hat{\varphi}_m = \hat{\varphi}_{\frac{a}{4},m}$.
\end{defn}
$\hat{\varphi}_{b,m}$ is a smooth function which takes the value $1$ on the disk $\vert x \vert < b-\frac{1}{m}$ and the value $0$ outside the disk $\vert x \vert < b+\frac{1}{m}$.  We shall now prove that the function is actually monotonic over the positive reals.
\begin{lem}\label{lem:incr}
Fix $b - \frac{1}{m} > 0$. Then $\hat{\varphi}_{b,m}$ is increasing over $(-\infty,0)$ and decreasing over $(0,\infty)$.
\end{lem}
\begin{proof}
Let $\gamma > 0$.  We calculate $\hat{\varphi}_{b,m}' (\gamma)$.  The numerator is 
\begin{eqnarray*}
&& \left(f(b + \frac{1}{m} - \gamma) + f(\gamma - b + \frac{1}{m})\right)\left(-f'(b + \frac{1}{m} - \gamma)\right) \\
&&  -  f(b + \frac{1}{m} - \gamma)\left(-f'(b + \frac{1}{m} - \gamma) + f'(\gamma - b + \frac{1}{m})\right),
\end{eqnarray*}
while the denominator is
\begin{equation*}
\left(f(b + \frac{1}{m} - \gamma) + f(\gamma - b + \frac{1}{m})\right)^2,
\end{equation*}
and thus
%\begin{equation*}
%= \frac{\left(f(b + \frac{1}{m} - \gamma) + f(\gamma - b + \frac{1}{m})\right)\left(-f'(b + \frac{1}{m} - \gamma)\right) - f(b + \frac{1}{m} - \gamma)\left(-f'(b + \frac{1}{m} - \gamma) + f'(\gamma - b + \frac{1}{m})\right)}{\left(f(b + \frac{1}{m} - \gamma) + f(\gamma - b + \frac{1}{m})\right)^2} 
%\end{equation*}
\begin{equation*}
\hat{\varphi}_{b,m}' (\gamma)= \frac{-\left(f(\gamma - b + \frac{1}{m}))f'(b + \frac{1}{m} - \gamma) + f(b + \frac{1}{m} - \gamma)f'(\gamma - b + \frac{1}{m})\right)}{\left(f(b + \frac{1}{m} - \gamma) + f(\gamma - b + \frac{1}{m})\right)^2}. 
\end{equation*}
For all $\gamma \in \widehat{\mathbb{R}}$, $f(\gamma) \geq 0$ and $f'(\gamma) = \frac{1}{\gamma^2}e^{-\frac{1}{\gamma}}\mathbbm{1}_{(0,\infty)} \geq 0$.  Hence for $\gamma \geq 0$, $\hat{\varphi}_{b,m}'(\gamma) \leq 0$.  Since $\hat{\varphi}_{b,m}$ is even, this implies that $\hat{\varphi}_{b,m}'(\gamma) \geq 0$ for $\gamma \leq 0$.
\end{proof}
\begin{thm}\label{thm:bump}
Let $0 < \alpha < \frac{1}{2}$ and $m > \max \{ \frac{6}{a}, \frac{2}{1-2a} \}$.  Define 
\begin{equation*}
\hat{\psi}_m = \hat{\varphi}_m (\gamma + \frac{3a}{4}) + \hat{\varphi}_m(\gamma - \frac{3a}{4}) \in C_c^\infty (\widehat{\mathbb{R}}).
\end{equation*}
Then $\psi_m \in \mathcal{S} (\mathbb{R})$ and $\mathcal{W}(\psi_m)$ forms a frame with bounds $A_m $ and $B_m $.  For all $m$, $A_m \leq \frac{1}{2}$, but as $m \rightarrow \infty$,  $B_m \rightarrow 1$.  
\end{thm}
\begin{proof}
As $m > \frac{4}{a}$, $\hat{\varphi}_m$ is well-defined, and it follows from the definition of $\hat{\varphi}_m$ that
\begin{equation*}
\supp \hat{\psi}_m = [-a- \frac{1}{m}, -\frac{a}{2} + \frac{1}{m}] \cup [\frac{a}{2} - \frac{1}{m}, a + \frac{1}{m}].
\end{equation*}
Since $m > \frac{2}{1-2a}$, $\supp \hat{\psi}_m \subset ( -\frac{1}{2}, \frac{1}{2})$.  By continuity, there exists $\epsilon > 0$ such that $[-a,-\frac{a}{2}) \cup [\frac{a}{2}, a) \subseteq\supp_\epsilon \hat{\psi}_m$.  Thus, it follows from Theorem \ref{thm:shrink} that $\mathcal{W} (\psi_m)$ forms a frame with lower frame bound $A_m = \underline{K}_{\psi_m}$ and upper frame bound $B_m =   \overline{K}_{\psi_m}$.
As $\hat{\psi}_m$ is even, it suffices to optimize $\kappa_{\psi_m}$ over any positive dyadic interval.  We shall use $[ \frac{a}{2} + \frac{1}{2},  a + \frac{1}{m} )$.  Since $m > \frac{5}{a}$, $\kappa_{\psi_m} (\gamma) = (\hat{\psi}_m (\gamma))^2  + (\hat{\psi}_m (\frac{\gamma}{2}))^2$ for $\gamma \in [\frac{a}{2} + \frac{1}{2m}, a + \frac{1}{m})$.  Also $m > \frac{6}{a}$ implies that $a - \frac{2}{m} > \frac{a}{2} + \frac{1}{m}$.  Hence, over $[ \frac{a}{2} + \frac{1}{2m}, a + \frac{1}{m})$, $\kappa_{\psi_m}$
\begin{eqnarray*}
 & = & (\hat{\psi}_m (\gamma) )^2 + (\hat{\psi}_m (\frac{\gamma}{2}))^2 \\
& = & \left\{ \begin{array}{lr}
(\hat{\psi}_m(\gamma))^2 + 0 & \textrm{for $ \frac{a}{2} + \frac{1}{2m} \leq \gamma < \frac{a}{2} + \frac{1}{m}$} \\
1 + 0 & \textrm{for $ \frac{a}{2} + \frac{1}{m} \leq \gamma < a - \frac{2}{m}$}\\
1 + (\hat{\psi}_m (\frac{\gamma}{2}))^2 & \textrm{for $a - \frac{2}{m} \leq \gamma < a - \frac{1}{m}$}\\
(\hat{\psi}_m (\gamma))^2 + (\hat{\psi}_m (\frac{\gamma}{2}))^2 & \textrm{for $a - \frac{1}{m} \leq \gamma < a + \frac{1}{m}$}\\
\end{array} \right. \\
& = & \left\{ \begin{array}{lr}
(\hat{\varphi}_m(\gamma - \frac{3a}{4}))^2 & \textrm{for $ \frac{a}{2} + \frac{1}{2m} \leq \gamma < \frac{a}{2} + \frac{1}{m}$} \\
1 & \textrm{for $ \frac{a}{2} + \frac{1}{m} \leq \gamma < a - \frac{2}{m}$}\\
1 + (\hat{\varphi}_m(\frac{\gamma}{2}- \frac{3a}{4}))^2 & \textrm{for $a - \frac{2}{m} \leq \gamma < a - \frac{1}{m}$}\\
(\hat{\varphi}_m (\gamma - \frac{3a}{4}))^2 + (\hat{\psi} (\frac{\gamma}{2} - \frac{3a}{4}))^2 & \textrm{for $a - \frac{1}{m} \leq \gamma < a + \frac{1}{m}$}\\
\end{array} \right. .
\end{eqnarray*} 
Note that
\begin{itemize}
\item $\gamma - \frac{3a}{4} < 0$ for $\frac{a}{2} + \frac{1}{2m} \leq \gamma < \frac{a}{2} + \frac{1}{m}$ since $m > \frac{4}{a}$,
\item $\frac{\gamma}{2} - \frac{3a}{4} < 0$ for $a - \frac{2}{m} \leq \gamma < a - \frac{1}{m}$,
\item $\gamma - \frac{3a}{4} > 0$ for $a - \frac{1}{m} \leq \gamma < a + \frac{1}{m}$ since $m > \frac{4}{a}$, and 
\item $\frac{\gamma}{2} - \frac{3a}{4} < 0$ for $a - \frac{1}{m} \leq \gamma < a + \frac{1}{m}$ since $m > \frac{2}{a}$.
\end{itemize}
Thus, $\kappa_{\psi_m}$ is increasing over $\frac{a}{2} + \frac{1}{2m} \leq \gamma \leq a - \frac{1}{m}$, but is not monotonic over $a - \frac{1}{m} < \gamma < a + \frac{1}{m}$.  Hence 
\begin{eqnarray*}
\min_{\frac{a}{2} + \frac{1}{2} \leq \gamma \leq a - \frac{1}{m}} \kappa_{\psi_m} (\gamma) & = & \kappa_{\psi_m} (\frac{a}{2} + \frac{1}{2}) =  \left(\hat{\varphi}_m ((\frac{a}{2} + \frac{1}{2m}) - \frac{3a}{4} )\right)^2  \\ 
& = &  \left( \hat{\varphi}_m (-\frac{a}{4} + \frac{1}{2m})\right)^2 =  \left( \frac{e^{-2m/3}}{e^{-2m/3} + e^{-2m}}\right)^2  \\
& = &  \left( \frac{1}{1+e^{-4m/3}}\right)^2,
\end{eqnarray*}
and
\begin{eqnarray*}
\max_{\frac{a}{2} + \frac{1}{2} \leq \gamma \leq a - \frac{1}{m}} \kappa_{\psi_m} (\gamma) & = & \kappa_{\psi_m} (a - \frac{1}{m}) \\
& = & 1 +\left(\hat{\varphi}_m (\frac{1}{2}(a - \frac{1}{m}) - \frac{3a}{4} )\right)^2 \\
& = & 1+  \left( \frac{e^{-2m}}{e^{-2m} + e^{-2m/3}}\right)^2 \\
& = & 1 + \left( \frac{1}{1+e^{4m/3}}\right)^2.
\end{eqnarray*}
Note that as $m \rightarrow \infty$, 
\begin{eqnarray*}
\min_{\frac{a}{2} + \frac{1}{2} \leq \gamma \leq a - \frac{1}{m}} \kappa_{\psi_m} (\gamma) & \rightarrow& 1 \\
\max_{\frac{a}{2} + \frac{1}{2} \leq \gamma \leq a - \frac{1}{m}} \kappa_{\psi_m} (\gamma) & \rightarrow& 1
\end{eqnarray*}
We shall now consider $\kappa_{\psi_m}$ over $(a - \frac{1}{m}, a+ \frac{1}{m})$.  We start by substituting $\gamma = a + \frac{t}{m}$, $t \in (-1,1)$ and expanding $\kappa_{\psi_m}$ over these values:
\begin{eqnarray*}
\kappa_{\psi_m} ( a + \frac{t}{m}) &  =  & \left( \hat{\varphi}_m (a + \frac{t}{m} - \frac{3a}{4})\right)^2  +  \left( \hat{\varphi}_m (\frac{a}{2} + \frac{t}{2m} - \frac{3a}{4})\right)^2 \\
&   =   & \left( \hat{\varphi}_m (\frac{a}{4} + \frac{t}{m})\right)^2  +  \left( \hat{\varphi}_m (-\frac{a}{4} + \frac{t}{2m})\right)^2 \\
&   =   & \left( \frac{e^{-m/(1-t)}}{e^{-m/(1-t)} + e^{-m/(1+t)}} \right)^2  +  \left( \frac{e^{-2m/(2+t)}}{e^{-2m/(2+t)} + e^{-2m/(2-t)}} \right)^2 \\
&   =   & \left( \frac{1}{1+ e^{2mt/(1-t^2)}} \right)^2  +   \left( \frac{1}{1+ e^{-4mt/(4-t^2)}} \right)^2
\end{eqnarray*}
At $t=0$, $\kappa_{\psi_m}$ takes the value $\left(\frac{1}{2}\right)^2 + \left(\frac{1}{2}\right)^2 = \frac{1}{2}$.  We claim that for any $0 < \delta < \frac{1}{2}$, $\kappa_{\psi_m} (a + \frac{t}{m})$ converges uniformly to $1$ over $[\delta, 1-\delta]$.  Choose an arbitrary $0 < \epsilon < 2$.  We claim that for any 
\begin{equation*}
m > \max \left\{ \frac{(1-\delta^2)\ln(\sqrt{\frac{2}{\epsilon}}-1)}{\delta} , \frac{(\delta^2 - 4)\ln(1-\sqrt{1-\frac{\epsilon}{2}})}{4\delta}\right\}
\end{equation*}
$\vert \kappa_{\psi_m} (a + \frac{t}{m}) - 1 \vert < \epsilon$ for all $t \in [\delta, 1-\delta]$.  A routine application of the triangle inequality yields
\begin{eqnarray*}
\vert \kappa_{\psi_m} (a + \frac{t}{m}) - 1 \vert & = & \left\vert \left( \frac{1}{1+ e^{2mt/(1-t^2)}} \right)^2 +  \left( \frac{1}{1+ e^{-4mt/(4-t^2)}} \right)^2 -1 \right\vert \\
& \leq & \left\vert  \left( \frac{1}{1+ e^{2mt/(1-t^2)}} \right)^2 \right\vert  + \left\vert \left( \frac{1}{1+ e^{-4mt/(4-t^2)}} \right)^2 -1 \right\vert. 
\end{eqnarray*}
Since $m >  \frac{(1-\delta^2)\ln(\sqrt{\frac{2}{\epsilon}}-1)}{\delta}$, for all $t  \in [\delta, 1-\delta]$,
\begin{eqnarray*}
&& \sqrt{\frac{2}{\epsilon}} - 1  < e^{2m\delta/(1-\delta^2)} \leq e^{2mt/(1-t^2)} \\
& \Rightarrow & \frac{2}{\epsilon} < (1+e^{2mt/(1-t^2)})^2 \\
& \Rightarrow & \left\vert  \left( \frac{1}{1+ e^{2mt/(1-t^2)}} \right)^2 \right\vert  < \frac{\epsilon}{2}.
\end{eqnarray*}
Since $m > \frac{(\delta^2 - 4)\ln(1-\sqrt{1-\frac{\epsilon}{2}})}{4\delta}$, 
\begin{eqnarray*}
&& 1 - \sqrt{1-\frac{2}{\epsilon}}  > e^{-4m\delta/(4-\delta^2)} \geq e^{-4mt/(4-t^2)} \\
& \Rightarrow & (e^{-4mt/(4-t^2)})^2 -2(e^{-4mt/(4-t^2)}) + \frac{\epsilon}{2} > 0 \\
& \Rightarrow & \vert (e^{-4mt/(4-t^2)})^2 -2(e^{-4mt/(4-t^2)}) \vert < \frac{\epsilon}{2} \\
& \Rightarrow & \vert  1- (1 + e^{-4mt/(4-t^2)})^2 \vert < \frac{\epsilon}{2} \\
& \Rightarrow & \left\vert  \frac{1- (1 + e^{-4mt/(4-t^2)})^2}{(1 + e^{-4mt/(4-t^2)})^2} \right\vert < \frac{\epsilon}{2} \\
& \Rightarrow & \left\vert  \frac{1}{(1 + e^{-4mt/(4-t^2)})^2} - 1\right\vert < \frac{\epsilon}{2} .
\end{eqnarray*}
Thus, $\vert \kappa_{\psi_m} (a + \frac{t}{m}) - 1 \vert < \epsilon$ for all $t \in [\delta, 1-\delta]$.  It is also true that for any $0 < \delta < \frac{1}{2}$, $\kappa_{\psi_m} (a + \frac{t}{m})$ converges uniformly to $1$ over $[-1+ \delta, -\delta]$.  The proof works in the same manner, except the triangle inequality is used in the following way
\begin{eqnarray*}
\vert \kappa_{\psi_m} (a + \frac{t}{m}) - 1 \vert &  =   & \left\vert \left( \frac{1}{1+ e^{2mt/(1-t^2)}} \right)^2 +  \left( \frac{1}{1+ e^{-4mt/(4-t^2)}} \right)^2 -1 \right\vert \\
&  \leq  & \left\vert  \left( \frac{1}{1+ e^{2mt/(1-t^2)}} \right)^2 -1 \right\vert  + \left\vert \left( \frac{1}{1+ e^{-4mt/(4-t^2)}} \right)^2 \right\vert. 
\end{eqnarray*}
Combining this convergence with our knowledge of the values $\kappa_{\psi_m}(0) = \frac{1}{2}$, $\kappa_{\psi_m} (a + \frac{1}{m})= \left( \frac{1}{1+e^{-4m/3}}\right)^2$, and $\kappa_{\psi_m} (a - \frac{1}{m})= 1 + \left( \frac{1}{1+e^{4m/3}}\right)^2$, we conclude that
\begin{equation*}
\lim_{m \rightarrow \infty} B_m = \lim_{m \rightarrow \infty} \overline{K}_{\psi_m} = 1.
\end{equation*}
\end{proof}
Using a result from \cite{BenKing09}, a corollary to Theorem~\ref{thm:bump} is that these smooth bump functions cannot be the result of convolutional smoothing with an approximate identity formed by dilations.  We reference the $1$-dimensional version of the theorem here.
\begin{thm} \label{thm:big}
For $0 < a < 1/2$, let $L \subseteq \widehat{\mathbb{R}}$ be the Parseval frame wavelet set $[-a,-\frac{a}{2}] \cup[\frac{a}{2},a]$.  Also let $g: \widehat{\mathbb{R}} \rightarrow \mathbb{R}$ satisfy the following conditions:
\begin{itemize}
\item[\textit{i.}] $\supp g \subseteq [-b,c]$, where $b, c >0$ and $\supp g$ contains a neighborhood of $0$;
\item[\textit{ii.}] $\int g(\gamma) d \gamma = 1$; and
\item[\textit{iii.}] $\displaystyle{0 < \int_{ [\frac{c}{2},{c}]} g(\gamma) d\gamma < 1}$ and $\displaystyle{0 < \int_{[-\frac{b}{2},{c}]} g(\gamma) d\gamma < 1}$.
\end{itemize}
Define $\hat{\psi}_{m} = \mathbbm{1}_L \ast g_{(m)}$.  For any
\begin{equation*}
m >  \max \bigg\lbrace \frac{2(b + c)}{a}, \frac{b + c}{1-2a}, \frac{4b + c}{a}, \frac{4c_i + b_i}{a} \bigg\rbrace ,
\end{equation*}
$\mathcal{W}(\psi_m)$ is a frame with frame bounds $A_m$ and $B_m$, and there exist $\alpha < 1$ and $\beta > 1$, both independent of $m$, such that $A_m \leq \alpha$ and $B_m \geq \beta$. These are called \emph{frame bound gaps}.
\end{thm}
\begin{cor}
Define $\psi_m$ as in Theorem~\ref{thm:bump}.  Then there does not exist a $g: \widehat{\mathbb{R}} \rightarrow \mathbb{R}$ such that $\hat{\psi}_m = \mathbbm{1}_{[-a,a]\cup[-a/2,a/2]} \ast g_{(m)}$, where $g_{(m)}(\cdot) = mg(m\cdot)$.
\end{cor}
\begin{proof}
Assume that such a $g$ exists. Since the building blocks of $\hat{\psi}_m$ satisfy $\supp \hat{\phi}_m = [-a/4-1/m,a/4+1/m]$, $\supp g = [-1,1]$.  Furthermore, since $\hat{\phi}_m$ is increasing over $(-\infty,0)$ (Lemma~\ref{lem:incr}), $g$ must be non-negative.  Thus such a $g$ satisfies that hypotheses of Theorem~\ref{thm:big}, and for large enough $m$, there exist frame bound gaps for $\mathcal{W}(\psi_m)$.  That is, the upper frame bound of $\mathcal{W}(\psi_m)$ cannot converge to one.  However, this contradicts Theorem~\ref{thm:bump}.
\end{proof}
Thus, these functions which are commonly used in mathematics are not the result of convolutional smoothing.

%%%%%%%%%%%%%%%%%%%%%%%
\section{Conclusion}\label{section:con}
%%%%%%%%%%%%%%%%%%%%%%%
Based on conversations with other authors in the wavelet set community and remarks in published papers, it seems that smoothing of higher dimensional wavelet set wavelets was assumed to be similar to smoothing in $\mathbb{R}$.  However, Corollary~38 in \cite{BenKing09} showed that convolutional smoothing of the Parseval frame wavelet set wavelets ($[-2a,2a]^d\backslash[-a,a]^d$) on the frequency domain yields systems with upper frame bounds which increase away from $1$ as the $d$ increases.  This theme is continued in this paper.  We see in Section~\ref{sec:fail} that natural generalizations of Bin Han's proof of existence of smooth Parseval wavelets in $L^2(\mathbb{R})$ to $L^2(\mathbb{R}^2)$ also fail.  We showed that smoothable wavelet sets do exist in higher dimensions, but these wavelet sets were created for the sole purpose of being smoothable.  Thus, the question remains whether there exist continuous functions $\hat{\psi}_n$ for which $\mathcal{W}(\psi_n)$ has frame bounds converging to $1$ and for which, say, $\Vert \mathbbm{1}_{[-a,a]^2 \backslash [-a/2,a/2]^2} - \hat{\psi}_n \Vert_{L^2(\widehat{\mathbb{R}}^2)}$ converges to $0$ as $n \rightarrow \infty$ for some $a < 1/2$.  Furthermore, we also explicitly construct smooth frame wavelets which have upper frame bounds converging to $1$ in Section \ref{sec:construct}.  Using results about frame bound gaps, this construction shows that basic functions used in differential topology are not the result of convolutional smoothing.

%%%%%%%%%%%%%%%%%%%%%%%
\section{Appendix}\label{section:app}
%%%%%%%%%%%%%%%%%%%%%%%
We include here the calculation-intensive proofs of the results in Section~\ref{sec:fail}.
\begin{proof}[Proposition~\ref{prop:h}]
We first rewrite $h_{(L_a;\delta,\frac{\delta}{2})}$ in terms of $\theta$ (from Lemma \ref{lem:theta}).
\begin{equation*}
\left\{ \begin{array}{lr}
\theta\left(\frac{2a-|y|}{\delta}\right) & \textrm{when $|x| \in [0,2a-\delta]$ and $|y| \in [2a - \delta, 2a + \delta]$} \\
\theta\left(\frac{2a-|y|}{\delta}\right)\theta\left(\frac{2a-|x|}{\delta}\right)& \textrm{when $|x|,|y| \in [2a - \delta, 2a + \delta]$} \\
\theta\left(\frac{2a-|x|}{\delta}\right) & \textrm{when $|y| \in [0,2a-\delta]$ and $|x| \in [2a - \delta, 2a + \delta]$} \\
1 & \textrm{when $(|x|,|y|)^T \in [0,2a-\delta]^2 \backslash [0,a + \frac{\delta}{2}]^2$} \\
1 - \theta\left(\frac{a-|y|}{\delta/2}\right) & \textrm{when $|x| \in [0,a-\frac{\delta}{2}]$ and $|y| \in [a - \frac{\delta}{2}, a + \frac{\delta}{2}]$} \\
1-\theta\left(\frac{a-|y|}{\delta/2}\right)\theta\left(\frac{a-|x|}{\delta/2}\right)& \textrm{when $|x|,|y| \in [a - \frac{\delta}{2}, a + \frac{\delta}{2}]$} \\
1-\theta\left(\frac{a-|x|}{\delta/2}\right) & \textrm{when $|y| \in [0,a-\frac{\delta}{2}]$ and $|x| \in [a - \frac{\delta}{2}, a + \frac{\delta}{2}]$} \\
0 & \textrm{otherwise}.
\end{array} \right.
\end{equation*}
We will prove the claim if we show that $\sum_{n\in\mathbb{Z}} h_{(L_a;\delta,\frac{\delta}{2})}^2(2^n x, 2^n y) < 1$ on a set of positive measure. Note that for $(x,y)^T \in [0,a-\frac{\delta}{2}]\times(a - \frac{\delta}{2}, a + \frac{\delta}{2}]$,
\begin{eqnarray*}
\sum_{n\in\mathbb{Z}} h_{(L_a;\delta,\frac{\delta}{2})}^2(2^n x, 2^n y) & = & h_{(L_a;\delta,\frac{\delta}{2})}^2(x,y) + h_{(L_a;\delta,\frac{\delta}{2})}^2(2x, 2y) \\
& = & \left(1-\theta\left(\frac{a-y}{\delta/2}\right)\right)^2 + \theta^2 \left(\frac{a-y}{\delta/2}\right) \\
& = & 1 + 2\theta\left(\frac{a-|y|}{\delta/2}\right)\left[\theta\left(\frac{a-|y|}{\delta/2}\right)-1\right] \\
& < & 1
\end{eqnarray*}
since $0 < \theta\left(\frac{a-|y|}{\delta/2}\right) < 1$ for $y > a - \frac{\delta}{2}$.
\end{proof}
\begin{proof}[Proposition~\ref{prop:g}]
We first compute the following for $x, y > 0$, making use of Lemma~\ref{lem:han}:
\begin{equation*}
g_{(L_a;\delta,\frac{\delta}{2})}^2(x,y) + g_{(L_a;\delta,\frac{\delta}{2})}^2(2x,2y)
\end{equation*}
\begin{equation*}
 \left\{ \begin{array}{lr}
\theta^2\left(\frac{2a-y}{\delta}\right)  &  (x,y)^T \in [0,2a-\delta]\times[2a - \delta, 2a + \delta] \\
\theta^2\left(\frac{2a-y}{\delta}\right)\theta^2\left(\frac{2a-x}{\delta}\right)  &  (x,y)^T \in [2a - \delta, 2a + \delta]^2 \\
\theta^2\left(\frac{2a-x}{\delta}\right) &  (x,y)^T \in [2a - \delta, 2a + \delta]\times[0,2a-\delta] \\
1 & \quad(x,y)^T \in [0,2a-\delta]^2 \backslash [0,a + \frac{\delta}{2}]^2\\
\theta^2\left(\frac{a-y}{\delta/2}\right)+ \theta^2\left(\frac{y-a}{\delta/2}\right)  &(x,y)^T \in [0,a-\frac{\delta}{2}]\times [a - \frac{\delta}{2}, a + \frac{\delta}{2}] \\
\theta^2\left(\frac{a-y}{\delta/2}\right)\theta^2\left(\frac{a-x}{\delta/2}\right)+\theta^2\left(\frac{y-a}{\delta/2}\right)\theta^2\left(\frac{x-a}{\delta/2}\right) & (x,y)^T \in [a - \frac{\delta}{2}, a + \frac{\delta}{2}]^2 \\
\theta^2\left(\frac{a-x}{\delta/2}\right) + \theta^2\left(\frac{x-a}{\delta/2}\right) &  (x,y)^T \in [a -\frac{\delta}{2}, a + \frac{\delta}{2}]\times[0,a-\frac{\delta}{2}]  \\
1  &  (x,y)^T \in [0,a-\frac{\delta}{2}]^2 \backslash [0,\frac{a}{2} + \frac{\delta}{4}]^2 \\
\theta^2\left(\frac{y-a/2}{\delta/4}\right) & (x,y)^T \in [0,\frac{a}{2}-\frac{\delta}{4}]\times [\frac{a}{2} - \frac{\delta}{4}, \frac{a}{2} + \frac{\delta}{4}] \\
\theta^2\left(\frac{y-a/2}{\delta/4}\right)\theta^2\left(\frac{x-a/2}{\delta/4}\right) &  (x,y)^T \in [\frac{a}{2} - \frac{\delta}{4}, \frac{a}{2} + \frac{\delta}{4}]^2\\
\theta^2\left(\frac{x-a/2}{\delta/4}\right) & (x,y)^T \in [\frac{a}{2} - \frac{\delta}{4}, \frac{a}{2} + \frac{\delta}{4}] \times [0,\frac{a}{2}-\frac{\delta}{4}] \\
0 & \textrm{otherwise}
\end{array} \right.
\end{equation*}
\begin{equation*}
 \left\{ \begin{array}{lr}
\theta^2\left(\frac{2a-y}{\delta}\right)& \quad(x,y)^T \in [0,2a-\delta]\times[2a - \delta, 2a + \delta] \\
\theta^2\left(\frac{2a-y}{\delta}\right)\theta^2\left(\frac{2a-x}{\delta}\right) & (x,y)^T \in [2a - \delta, 2a + \delta]^2\\
\theta^2\left(\frac{2a-x}{\delta}\right) & \quad(x,y)^T \in [2a - \delta, 2a + \delta]\times[0,2a-\delta] \\
1 &  (x,y)^T \in [0,2a-\delta]^2 \backslash [0,a + \frac{\delta}{2}]^2 \\
1 & (x,y)^T \in [0,a-\frac{\delta}{2}]\times [a - \frac{\delta}{2}, a + \frac{\delta}{2}] \\
\theta^2\left(\frac{a-y}{\delta/2}\right)\theta^2\left(\frac{a-x}{\delta/2}\right) +\theta^2\left(\frac{y-a}{\delta/2}\right)\theta^2\left(\frac{x-a}{\delta/2}\right) & (x,y)^T \in [a - \frac{\delta}{2}, a + \frac{\delta}{2}]^2 \\
1  &  (x,y)^T \in [a -\frac{\delta}{2}, a + \frac{\delta}{2}]\times[0,a-\frac{\delta}{2}]  \\
1 & (x,y)^T \in [0,a-\frac{\delta}{2}]^2 \backslash [0,\frac{a}{2} + \frac{\delta}{4}]^2 \\
\theta^2\left(\frac{y-a/2}{\delta/4}\right) & (x,y)^T \in [0,\frac{a}{2}-\frac{\delta}{4}]\times [\frac{a}{2} - \frac{\delta}{4}, \frac{a}{2} + \frac{\delta}{4}] \\
\theta^2\left(\frac{y-a/2}{\delta/4}\right)\theta^2\left(\frac{x-a/2}{\delta/4}\right) & (x,y)^T \in [\frac{a}{2} - \frac{\delta}{4}, \frac{a}{2} + \frac{\delta}{4}]^2 \\
\theta^2\left(\frac{x-a/2}{\delta/4}\right) & (x,y)^T \in [\frac{a}{2} - \frac{\delta}{4}, \frac{a}{2} + \frac{\delta}{4}] \times [0,\frac{a}{2}-\frac{\delta}{4}] \\
0 & \textrm{otherwise}.
\end{array} \right.
\end{equation*}
%%%%%%%%%%%%%%%%%%%%%%%%%%%%%%%%%%%%%%%%
Continuing inductively we obtain $\sum_{n\in\mathbb{Z}} g_{(L_a;\delta,\frac{\delta}{2})}^2(2^n x, 2^n y)$
\begin{equation*}
 = \left\{ \begin{array}{lr}
0 & \textrm{when $x=y=0$} \\
\theta^2\left(\frac{a-2^my}{\delta/2}\right)\theta^2\left(\frac{a-2^mx}{\delta/2}\right) & (2^mx,2^my)^T \in [a - \frac{\delta}{2}, a + \frac{\delta}{2}]^2; m \in \mathbb{Z} \\
\quad + \theta^2\left(\frac{2^my-a}{\delta/2}\right)\theta^2\left(\frac{2^mx-a}{\delta/2}\right) &
\\1 & \textrm{otherwise.}
\end{array} \right. 
\end{equation*}
We would like to show that $\theta^2\left(\frac{a-y}{\delta/2}\right)\theta^2\left(\frac{a-x}{\delta/2}\right) +\theta^2\left(\frac{y-a}{\delta/2}\right)\theta^2\left(\frac{x-a}{\delta/2}\right)$ does not take the value $1$ for almost all $(x,y)^T \in [a - \frac{\delta}{2}, a + \frac{\delta}{2}]^2$. Assume that $\frac{1}{2} < \beta < 1$ and $0 < \alpha < 1$. Then $\frac{\beta}{2\beta -1} > 1$, implying
\begin{eqnarray*}
 (\alpha \neq \frac{\beta}{2\beta - 1})  &\Rightarrow&  (2\alpha\beta - \alpha \neq \beta)\\
  & \Rightarrow & (1 + 2\alpha\beta - \alpha -\beta \neq 1) \\
  &\Rightarrow & (\alpha\beta + (1 - \alpha)(1 - \beta) \neq 1)
\end{eqnarray*}
It follows from the intermediate value theorem and continuity that the measure of $E = \{ (x,y)^T \in [a - \frac{\delta}{2}, a + \frac{\delta}{2}]^2 : \frac{1}{2} < \theta^2\left(\frac{a-y}{\delta/2}\right) < 1, 0 < \theta^2 \left(\frac{a-x}{\delta/2}\right) < 1 \}$ is positive. So for $(x,y)^T \in E$,
\begin{equation*}
\theta^2\left(\frac{a-y}{\delta/2}\right)\theta^2\left(\frac{a-x}{\delta/2}\right) + \theta^2\left(\frac{y-a}{\delta/2}\right)\theta^2\left(\frac{x-a}{\delta/2}\right)
 \end{equation*}
 \vspace{-5 mm}
\begin{eqnarray*}
& = & \theta^2\left(\frac{a-y}{\delta/2}\right)\theta^2\left(\frac{a-x}{\delta/2}\right) + \left(1 - \theta^2\left(\frac{a-y}{\delta/2}\right)\right)\left(1 - \theta^2\left(\frac{a-x}{\delta/2}\right)\right) \\
& \neq & 1
\end{eqnarray*}
\end{proof}
\begin{proof}[Proposition~\ref{prop:f}]
We calculate the sum
\begin{equation*}
\hat{\psi}_\delta^2(\vec{x}) + \hat{\psi}_\delta^2(2\vec{x}) 
\end{equation*}
%{\small
\begin{equation*}
 =  \left\{ \begin{array}{lr}
\theta^2\left(\frac{2a-y}{\delta}\right) & (x,y)^T \in [0,2a-\delta]\times[2a - \delta, 2a + \delta] \\
\theta^2\left(\frac{4a-x-y-\delta}{\delta}\right) &(x,y)^T \in [2a - \delta, 2a + \delta]^2 \\
\theta^2\left(\frac{2a-x}{\delta}\right) & (x,y)^T \in [2a - \delta, 2a + \delta]\times[0,2a-\delta] \\
1 & (x,y)^T \in [0,2a-\delta]^2 \backslash [0,a + \frac{\delta}{2}]^2 \\
1 & (x,y)^T \in [a - \frac{\delta}{2}, a + \frac{\delta}{2}]^2; 2a \leq x + y 2a+\delta \\
\theta^2\left(\frac{a-y}{\delta/2}\right) + \theta^2\left(\frac{y-a}{\delta/2}\right) & (x,y)^T \in [0,a-\frac{\delta}{2}]\times [a - \frac{\delta}{2}, a + \frac{\delta}{2}] \\
\theta^2\left(\frac{x+y-2a+\delta/2}{\delta/2}\right) +& (x,y)^T \in [a - \frac{\delta}{2}, a + \frac{\delta}{2}]^2\\
\quad \theta^2\left(\frac{2a-x-y-\delta/2}{\delta/2}\right) & \quad\textrm{ for $2a-\delta \leq x+y\leq 2a$} \\
\theta^2\left(\frac{a-x}{\delta/2}\right) + \theta^2\left(\frac{x-a}{\delta/2}\right) &(x,y)^T \in [a -\frac{\delta}{2}, a + \frac{\delta}{2}]\times[0,a-\frac{\delta}{2}]  \\
1 & (x,y)^T \in [0,a-\frac{\delta}{2}]^2 \backslash [0,\frac{a}{2} + \frac{\delta}{4}]^2 \\
1 & (x,y)^T \in [\frac{a}{2} - \frac{\delta}{4}, \frac{a}{2} + \frac{\delta}{4}]^2 \textrm{ for $a \leq x + y \leq a+\frac{\delta}{2}$} \\
\theta^2\left(\frac{y-a/2}{\delta/4}\right) & (x,y)^T \in [0,\frac{a}{2}-\frac{\delta}{4}]\times [\frac{a}{2} - \frac{\delta}{4}, \frac{a}{2} + \frac{\delta}{4}] \\
\theta^2\left(\frac{x+y-a+\delta/4}{\delta/4}\right) &(x,y)^T \in [\frac{a}{2} - \frac{\delta}{4}, \frac{a}{2} + \frac{\delta}{4}]^2 \textrm{ for $a - \frac{\delta}{2} \leq x+y \leq a$} \\
\theta^2\left(\frac{x-a/2}{\delta/4}\right)&(x,y)^T \in [\frac{a}{2} - \frac{\delta}{4}, \frac{a}{2} + \frac{\delta}{4}] \times [0,\frac{a}{2}-\frac{\delta}{4}] \\
0 & \textrm{otherwise},
\end{array} \right. 
\end{equation*}
%}
%{\small
\begin{eqnarray*}
& = & \left\{ \begin{array}{lr}
\theta^2\left(\frac{2a-y}{\delta}\right)& (x,y)^T \in [0,2a-\delta]\times[2a - \delta, 2a + \delta] \\
\theta^2\left(\frac{4a-x-y-\delta}{\delta}\right)&(x,y)^T \in [2a - \delta, 2a + \delta]^2 \\
\theta^2\left(\frac{2a-x}{\delta}\right)& (x,y)^T \in [2a - \delta, 2a + \delta]\times[0,2a-\delta] \\
1 &(x,y)^T \in [0,2a-\delta]^2 \backslash [0,\frac{a}{2} + \frac{\delta}{4}]^2 \\
1 &(x,y)^T \in [\frac{a}{2} - \frac{\delta}{4}, \frac{a}{2} + \frac{\delta}{4}]^2 \textrm{ for $a \leq x+y \leq a + \frac{\delta}{2}$} \\
\theta^2\left(\frac{y-a/2}{\delta/4}\right) &x,y)^T \in [0,\frac{a}{2}-\frac{\delta}{4}]\times [\frac{a}{2} - \frac{\delta}{4}, \frac{a}{2} + \frac{\delta}{4}] \\
\theta^2\left(\frac{x+y-a+\delta/4}{\delta/4}\right) & (x,y)^T \in [\frac{a}{2} - \frac{\delta}{4}, \frac{a}{2} + \frac{\delta}{4}]^2 \textrm{ for $a - \frac{\delta}{2} \leq x+y \leq a$} \\
\theta^2\left(\frac{x-a/2}{\delta/4}\right) & (x,y)^T \in [\frac{a}{2} - \frac{\delta}{4}, \frac{a}{2} + \frac{\delta}{4}] \times [0,\frac{a}{2}-\frac{\delta}{4}] \\
0 & \textrm{otherwise},
\end{array} \right. \\
& = & f_{([-2a,2a]^2\backslash[-\frac{a}{2},\frac{a}{2}]^2;\delta,\frac{\delta}{4})}^2(\vec{x}),
\end{eqnarray*}
%}
as desired.
\end{proof}

\section{Acknowledgements}
While writing this paper, the author was supported in part by a Department of Education GAANN Fellowship and a University of Maryland Graduate School Ann G. Wylie Dissertation Fellowship.

%\bibliography{thesis}{}
%\bibliographystyle{naturemag}

\end{document}